\documentclass[oneside,english]{amsart}
\usepackage[T1]{fontenc}
\usepackage[latin9]{inputenc}
\usepackage{array}
\usepackage{float}
\usepackage{multirow}
\usepackage{amstext}
\usepackage{amsthm}
\usepackage{amssymb}
\usepackage{setspace}
\usepackage{textcomp}
\makeatletter

\providecommand{\tabularnewline}{\\}

\numberwithin{equation}{section}
\numberwithin{figure}{section}
  \theoremstyle{plain}
  \newtheorem*{conjecture*}{\protect\conjecturename}
\theoremstyle{plain}
\newtheorem{thm}{\protect\theoremname}
  \theoremstyle{plain}
  \newtheorem{lem}[thm]{\protect\lemmaname}
  \theoremstyle{definition}
  \newtheorem{defn}[thm]{\protect\definitionname}
  \theoremstyle{plain}
  \newtheorem{conjecture}[thm]{\protect\conjecturename}
  \theoremstyle{remark}
  \newtheorem{rem}[thm]{\protect\remarkname}
  \theoremstyle{remark}
  \newtheorem*{rem*}{\protect\remarkname}
  \theoremstyle{plain}
  \newtheorem{prop}[thm]{\protect\propositionname}

\date{}
\usepackage{tensor}

\newenvironment{conjbis}
{
  \edef\thmlabel{\thethm}
  \renewcommand{\thethm}{\thmlabel$'$}%
  \addtocounter{thm}{-1}%
   \begin{conjecture}}
  {\end{conjecture}}

\newcommand{\ad}{\operatorname{ad}}
\newcommand{\Ad}{\operatorname{Ad}}
\newcommand{\cpt}{\operatorname{cpt}}
\newcommand{\sems}{\operatorname{ss}}
\newcommand{\scon}{\operatorname{sc}}
\newcommand{\tor}{\operatorname{tor}}
\newcommand{\der}{\operatorname{der}}
\newcommand{\Hh}{\operatorname{H}}
\newcommand{\SL}{\operatorname{SL}}
\newcommand{\Hom}{\operatorname{Hom}}

\makeatother

\usepackage{babel}
  \providecommand{\conjecturename}{Conjecture}
  \providecommand{\definitionname}{Definition}
  \providecommand{\lemmaname}{Lemma}
  \providecommand{\propositionname}{Proposition}
  \providecommand{\remarkname}{Remark}
\providecommand{\theoremname}{Theorem}

\begin{document}

\title{Generic Representations in $L$-packets}

\author{Manish Mishra}

\curraddr{The Ruprecht-Karls-Universität Heidelberg, Im Neuenheimer Feld 288
Heidelberg 69120}

\email{manish.mishra@gmail.com}


\begin{abstract}
We give the details of the construction of a map to restate a conjectural
expression about adjoint group action on generic representations in
$L$-packets. We give an application of the construction to give another
proof of the classification of the Knapp-Stein $R$-group associated
to a unitary unramified character of a torus. Finally we prove the
conjecture for unramified $L$-packets.
\end{abstract}

\maketitle
\footnote{Research supported by ERC AdG Grant 247049}

\section{Introduction}

Let $G$ be a quasi-split connected reductive group defined over a
local field $k$ of characteristic zero and let $Z$ be the center
of $G$. Let $B$ be a $k$-Borel subgroup of $G$ and let $T$ be
a maximal $k$-torus in $B$. Let $U$ be the unipotent radical of
$B$. A character $\psi:U(k)\rightarrow\mathbb{C}^{\times}$ is called
\textit{generic} if the stabilizer of $\psi$ in $T(k)$ is exactly
the center $Z(k)$. An irreducible admissible representation $\pi$
of $G$ is called \textit{generic}\textbf{ }($\psi$-\textit{generic})
if there exists a generic character $\psi$ of $U(k)$ such that $\Hom_{G(k)}(\pi,\mbox{Ind}_{U(k)}^{G(k)}\psi)\neq0$. 

The conjectural \textit{local Langlands program} partitions the irreducible
admissible representations of $G$ into finite sets known as $L$-\textit{packets}.
Each $L$-packet is expected to be parametrized by an arithmetic object
called the \textit{Langlands parameter, }which is an \textit{admissible
homomorphism} from the \textit{Weil-Deligne} group $W_{k}^{\prime}$
of $k$ to the $L$-\textit{group }$^{L}G$ of $G$. See \cite{Borel}
for the definitions and statements.

To each Langlands parameter $\varphi$, one can associate a finite
group $\mathcal{S}_{\varphi}$ (see \cite[Section 1, eq. (1.1)]{arthur06}).
It is expected that the associated $L$-packet $\Pi_{\varphi}$ is
parametrized by the irreducible representations $\widehat{\mathcal{S}_{\varphi}}$
of $\mathcal{S}_{\varphi}$ \cite[Section 1]{arthur06}. The parametrization
will depend on the choice of a Whittaker datum for $G$, which is
a $G(k)$-conjugacy class of pairs $(B,\psi)$, where $\psi$ is a
generic character of $U(k)$. When $\Pi_{\varphi}$ is generic, i.e.,
it has a generic representation, the $\psi$-generic representation
in $\Pi_{\varphi}$ is then required to correspond to the trivial
representation of $\mathcal{S}_{\varphi}$. The parametrization is
also expected to satisfy certain conjectural endoscopic character
identity \cite{kaletha}. When $\varphi$ is a \textit{tempered }parameter,
i.e., a parameter whose image projects onto a relatively compact subset
of the complex dual $\hat{G}$ of $G$\textit{, }Shahidi's \textit{tempered
$L$-packet conjecture} \cite[\S 9]{sha90} states that $\Pi_{\varphi}$
must be generic.

Let $\Gamma_{k}$ be the absolute Galois group of $k$ and write H$^{1}(k,-)$
for H$^{1}(\Gamma_{k},-).$ In Section \ref{sec:Construction}, we
construct a map $\gamma_{\varphi}:R_{\varphi}:=\pi_{0}(Z_{\hat{G}}(\mbox{Im}(\varphi)))\rightarrow\mbox{H}^{1}(k,X(Z))$,
where $X(Z)$ is the character lattice of $Z$ and $\varphi$ is any
Langlands parameter. Using Tate duality, we get the dual map $\hat{\gamma}_{\varphi}:\mbox{H}^{1}(k,Z)\rightarrow\widehat{R_{\varphi}}$,
where $\widehat{R_{\varphi}}$ is the set of irreducible representations
of $R_{\varphi}$. Let $p:t\in T\rightarrow\bar{t}\in T_{\ad}:=T/Z$
be the adjoint morphism. The finite abelian group $T_{\ad}(k)/p(T(k))\hookrightarrow\Hh^{1}(k,Z)$
acts simply transitively on the set of $T(k)$-orbits of generic characters
\cite[\S 3]{DeRe10}. The map $\zeta_{\varphi}:=\hat{\gamma}_{\varphi}|T_{\ad}(k)/p(T(k))$
factors through $\widehat{\mathcal{S}_{\varphi}}$ (see \cite[Sec. 9(4)]{ggp}
, also \cite[Sec. 3]{kaletha}). 

Now fix a parametrization $\rho\in\widehat{\mathcal{S}_{\varphi}}\mapsto\pi_{\rho}\in\Pi_{\varphi}$
by making the choice of a Whittaker datum. The following is a version
of the conjecture in \cite[Sec. 9(3)]{ggp} for generic $L$-packets.
\begin{conjecture*}
A representation $\pi_{\rho}\in\Pi_{\varphi}$ is $\psi$-generic
iff $\pi_{t\cdot\rho}$ is $t\cdot\psi$ generic for all $t\in T_{\ad}(k)$,
where $t\cdot\rho:=\rho\otimes\zeta_{\varphi}(t)$.
\end{conjecture*}
The map $\hat{\gamma}_{\varphi}$ was constructed in \cite{Kuo10}
in a very special case ($G$ split semisimple and $\varphi$ is the
parameter associated to a unitary character of $T(k)$). For depth
zero supercuspidal $L$-packets, the conjecture follows from \cite{DeRe10}.
When $G$ is semi-simple and split and the $L$-packet is formed by
the constituents of a unitary principal series, the conjecture follows
from \cite{Kuo02}. In \cite{kaletha}, Kaletha gives a proof of the
above conjecture for classical groups using very general arguments. 

Now let $G$ be unramified and let $\varphi$ be the parameter associated
to a unitary unramified character $\lambda$ of $T(k).$ The construction
of the map $\gamma_{\varphi}$ allows one to obtain a nice description
of the group $R_{\varphi}$ as a certain subgroup of an extended affine
Weyl group (Proposition \ref{thm:R-gp}). Using this, in Theorem \ref{thm:mykeys},
we obtain in a conceptual and uniform way, the classification of the
Knapp-Stein $R$-group associated to $\lambda$. This kind of classification
was obtained by Keys \cite[\S 3]{keys} in a case by case manner.
For split groups, using different methods, another way of getting
the classification obtained by Keys was recently given by Kamran and
Plymen \cite{KP13}. Our situation is more general and we also describe
the isomorphism, which has a simple description.

Finally in Theorem \ref{thm:Un_gen_struct}, we prove the conjecture
for \textit{unramified $L$-packets} (see Sec. \ref{sec:packet_struct}).
We do not assume the packet to be tempered.

\section{\label{sec:Prelem}Preliminaries}

\subsection{\label{sub:GpCoho}Group Cohomology}

For details about this subsection, see \cite[Ch. 5]{serre}.

Let $\Gamma$ be a topological group and let 
\begin{equation}
1\rightarrow A\rightarrow B\rightarrow C\rightarrow1\label{eq:1}
\end{equation}
 be a short exact sequence of $\Gamma$-groups. Assume that $A$ is
central subgroup of $B$. Then $C$ acts on $B$ by inner automorphisms
and it acts trivially on $A$. Let $\gamma:\Gamma\rightarrow C$ be
a co-cycle in $C$, i.e., it satisfies the relation $\gamma(ab)=\gamma(a)\tensor*[^{a}]{\gamma}{}(b)$
for all $a,b\in\Gamma$. By twisting the short exact sequence in (\ref{eq:1})
by $\gamma$, we get another short exact sequence
\[
1\rightarrow A\rightarrow\tensor*[_{\gamma}]{B}{}\rightarrow\tensor*[_{\gamma}]{C}{}\rightarrow1
\]
 From this we get a long exact cohomology sequence
\begin{eqnarray*}
1 & \rightarrow & \mbox{H}^{0}(\Gamma,A)\rightarrow\mbox{H}^{0}(\Gamma,\tensor*[_{\gamma}]{B}{})\rightarrow\mbox{H}^{0}(\Gamma,\tensor*[_{\gamma}]{C}{})\rightarrow\mbox{H}^{1}(\Gamma,A)\rightarrow\\
 & \rightarrow & \mbox{H}^{1}(\Gamma,\tensor*[_{\gamma}]{B}{})\rightarrow\mbox{H}^{1}(\Gamma,\tensor*[_{\gamma}]{C}{}).
\end{eqnarray*}

\subsection{\label{sub:Affine}Affine roots and affine transformations}

\subsubsection{\label{subsub:Omeg}The group $\Omega$.}

Let $\Psi=(X,R,\Delta,\check{X},\check{R},\check{\Delta})$ be a based
root datum in the sense of \cite[1.9]{springer79}. So $X$ and $\check{X}$
are free abelian groups in duality by a pairing $X\times\check{X}\rightarrow\mathbb{Z}$,
$R$ is a root system in the vector space $Q\otimes\mathbb{R}$, where
$Q$ is the root lattice, $\check{R}$ is the set of co-roots, $\Delta\subset R$
is a basis and $\check{\Delta}$ is the dual basis. Let $W=W(\Psi)$
be the Weyl group. The set $\Delta$ determines an alcove $C$ in
$V:=X\otimes\mathbb{R}$ in the following way. Let $\check{\Delta}=\{\check{\alpha_{1}},\ldots,\check{\alpha_{l}}\}$
and let $\check{\beta}=\sum_{i=1}^{l}n_{i}\check{\alpha_{i}}$ be
the highest co-root. Then $C$ is the alcove in $V$ defined by $C=\{x\in V:\check{\alpha_{0}}(x)\geq0,\ldots,\check{\alpha_{l}}(x)\geq0\}$,
where $\check{\alpha_{0}}=1-\check{\beta}$. Let $\tilde{W}=W\ltimes X$
and $\tilde{W}^{\circ}=W\ltimes Q$. Let $\Omega$ be the stabilizer
of $C$ in $\tilde{W}$. Then $\tilde{W}=\Omega\ltimes\tilde{W}^{\circ}$. 

Now assume that $\Psi$ is \textit{semisimple} \cite[1.1]{springer79}
and $R$ is an irreducible root system in $V$. Let $c_{0}$ be the
\textit{weighted barycenter} of $C$, characterized by the equations
$\check{\alpha_{i}}(c_{0})=1/h$ for $i=0,\ldotp,l$, where $h$ is
the Coxeter number. For any $w\in W$, let $\tilde{w}$ be the affine
map $x\in V\mapsto w(x-c_{0})+c_{0}$. It is the unique affine map
fixing $c_{0}$ with tangent part $w$. The following lemmas follow
from \cite[Lemma 6.2]{AnYu}.
\begin{lem}
\label{lem:yu1}For any $w\in W$, the following are equivalent:\end{lem}
\begin{enumerate}
\item $\tilde{w}\in\tilde{W}$.
\item $\tilde{w}\in\Omega$.\end{enumerate}
\begin{lem}
\label{lem:yu2}There is an isomorphism $\iota:\Omega\rightarrow X/Q$
defined by any of the following ways\end{lem}
\begin{enumerate}
\item $\iota(\tilde{w})=(w^{-1}-1)c_{0}+Q$.
\item The natural projection $\tilde{W}\rightarrow\tilde{W}/\tilde{W}^{\circ}=X/Q$
restricted to $\Omega$. 
\end{enumerate}

\subsubsection{\label{subsub:Basedroot}Based root datum}

Let $\Psi=(X,R,\Delta,\check{X},\check{R},\check{\Delta})$ be a reduced
based root datum. Let $\theta$ be a finite group acting on $\Psi$.
In \cite{Yu}, Jiu-Kang Yu defines the following $6$-touple $\underline{\Psi}=(\underline{X},\underline{R},\underline{\Delta},\check{\underline{X}},\check{\underline{R}},\check{\underline{\Delta}})$:
\begin{eqnarray*}
\underline{X} & = & X_{\theta}/\mbox{torsion},\\
\check{\underline{X}} & = & \check{X}^{\theta},\\
\underline{R} & = & \{\underline{a}:a\in R\},\qquad\mbox{ where }\underline{a}:=a|_{\check{\underline{X}}}\\
\check{\underline{R}} & = & \{\check{\alpha}:\alpha\in\underline{R}\},\\
\underline{\Delta} & = & \{\underline{a}:a\in\Delta\},\\
\check{\underline{\Delta}} & = & \{\check{\alpha}:\alpha\in\underline{\Delta}\}.
\end{eqnarray*}

The explanation for the defining formulas is as follows. We first
note that $\underline{X}$ and $\check{\underline{X}}$ are free abelian
groups, dual to each other under the canonical pairing $(\underbar{\ensuremath{x}},y)\mapsto<x,y>$,
for $\underbar{\ensuremath{x}}\in\underline{X}$, $y\in\check{\underline{X}}\subset\check{X}$,
where $x$ is any preimage of $\underbar{\ensuremath{x}}$ in $X$.
Define $\check{\alpha}$ for $\alpha\in\underline{R}$ as follows:
\begin{equation}
\check{\alpha}=\begin{cases}
\underset{a\in R:a|_{\check{\underline{X}}}=\alpha}{\sum\check{a},} & \mbox{if }2\alpha\notin\underline{R}\\
\underset{a\in R:a|_{\check{\underline{X}}}=\alpha}{2\sum\check{a},} & \mbox{if }2\alpha\in\underline{R}
\end{cases}
\end{equation}
In \cite{Yu}, Jiu-Kang Yu proves the following.
\begin{thm}
\label{thm:yu}\textnormal{[Jiu-Kang Yu]}The 6-touple $\underline{\Psi}=(\underline{X},\underline{R},\underline{\Delta},\check{\underline{X}},\check{\underline{R}},\check{\underline{\Delta}})$,
with the canonical pairing between $\underline{X}$ and $\check{\underline{X}}$
and the correspondence $\underline{R}\rightarrow\check{\underline{R}}$,
$\alpha\mapsto\check{\alpha}$, is a based root datum. Moreover, the
homomorphism $W(\Psi)^{\sigma}\rightarrow\mathrm{\textnormal{\textbf{GL}}}(\check{\underline{X}})$,
$w\mapsto w|_{\check{\underline{X}}}$ is injective and the image
is $W(\underbar{\ensuremath{\Psi}})$.
\end{thm}
The above Theorem for simply connected groups is proved in \cite[Sec. 3.3]{reeder2008torsion}.

\section{\label{sec:Construction}A construction and a conjecture}

\subsection{Construction}

Let $G$ be a quasi-split group defined over a local field $k$ of
characteristic zero. Let $T$ be a maximal $k$-torus of $G$ which
is contained in a $k$-Borel subgroup $B$. Let $\hat{G}_{\scon}$
be the simply connected cover of the derived group $\hat{G}_{\der}$
of $\hat{G}$, where $\hat{G}$ is the complex dual of $G$. Let $\hat{T}\subset\hat{G}$
be the torus dual to $T$ and $\hat{T}_{\scon}$ be the pull back
of $(\hat{T}\cap\hat{G}_{\der})^{\circ}$ via $\hat{G}_{\scon}\rightarrow\hat{G}_{\der}$.
Let $X=X(T)$ (resp. $\check{X}=\check{X}(T)$) denote the group of
characters (resp. co-characters) of $T$. Let $Z$ be the center of
$G$ and let $\hat{\mathfrak{z}}$ be the Lie algebra of the center
$\hat{Z}$ of $\hat{G}$. Then $\tilde{G}:=\hat{G}_{\scon}\times\hat{\mathfrak{z}}$
is the topological universal cover of $\hat{G}$. We have a short
exact sequence
\begin{equation}
1\rightarrow\pi_{1}(\hat{G})\rightarrow\tilde{G}\rightarrow\hat{G}\rightarrow1,\label{eq:1-1}
\end{equation}
 where $\pi_{1}(\hat{G})$ is the topological fundamental group of
$\hat{G}$. Let $Q$ denote the root lattice. Then from \cite[2.15]{springer79},
\begin{equation}
X(Z)\cong X/Q.\label{eq:2}
\end{equation}
The algebraic fundamental group of $\hat{G}$ is $\check{X}(\hat{T})/\check{X}(\hat{T}_{\scon})=X/Q$.
Since $\hat{G}$ is a complex algebraic group, its algebraic fundamental
group is the same as its topological fundamental group (see \cite{Boro}).
Therefore 
\begin{equation}
X/Q\cong\pi_{1}(\hat{G}).\label{eq:3}
\end{equation}
Let $W_{k}$ (resp. $\Gamma_{k}$) denote the Weil group (resp. absolute
Galois group) of $k$. Define $W_{k}^{\prime}:=W_{k}$ if $k$ is
archimedean and $W_{k}^{\prime}:=W_{k}\times\SL(2,\mathbb{C})$ if
$k$ is non-archimedean. $W_{k}^{\prime}$ is called the Weil-Deligne
group of $k$. Let $\varphi:W_{k}^{\prime}\rightarrow\tensor*[^{L}]{G}{}$
be a Langlands parameter (see \cite[Sec. 8.2]{Borel1979}). View $\varphi$
as an admissible homomorphism. Then $\varphi$ determines a co-cycle
$\phi|_{W_{k}}:W_{k}\rightarrow\tensor[^{L}]{G}{}\rightarrow\hat{G}$.
We can twist the exact sequence (\ref{eq:1-1}) by the co-cycle $\phi$
(see Section \ref{sub:GpCoho}). Then using the isomorphism $X(Z)\cong\pi_{1}(\hat{G})$,
we get 
\[
\tilde{\gamma}:\mbox{H}^{0}(W_{k},{}_{\phi}\hat{G})\rightarrow\mbox{H}^{1}(W_{k},X(Z)).
\]
Since $\mbox{H}^{0}(W_{k},{}_{\phi}\hat{G})\supset Z_{\hat{G}}(\mbox{Im}(\varphi))$,
by restriction this induces 
\[
\tilde{\gamma}^{\prime}:Z_{\hat{G}}(\mbox{Im}(\varphi))\rightarrow\mbox{H}^{1}(W_{k},X(Z)).
\]
Since this map is continuous and $\mbox{H}^{1}(W_{k},X(Z))$ is discrete,
ker$(\tilde{\gamma}^{\prime})\supset(Z_{\hat{G}}(\mbox{Im}(\varphi)))^{\circ}$.
Thus we get a map
\begin{equation}
\gamma_{\varphi}^{\prime}:R_{\varphi}:=\pi_{0}(Z_{\hat{G}}(\mbox{Im}(\varphi)))\rightarrow\mbox{H}^{1}(W_{k},X(Z)).\label{eq:4}
\end{equation}
 Since $R_{\varphi}$ is finite, $\gamma_{\varphi}^{\prime}$ induces
\[
\gamma_{\varphi}^{\prime\prime}:R_{\varphi}\rightarrow H^{1}(W_{k},X(Z))^{\tor}.
\]
By \cite[Theorem 4.1.3 (ii)]{Karpuk}, we have a functorial isomorphism
\[
\mbox{H}^{1}(W_{k},X(Z))^{\tor}=\mbox{H}^{1}(k,X(Z)).
\]
Here we are abbreviating H$^{1}(\Gamma_{k},-)$ by the notation H$^{1}(k,-)$.
We thus get a map
\begin{equation}
\gamma_{\varphi}:R_{\varphi}\rightarrow\mbox{H}^{1}(k,X(Z)).\label{eq:4-1}
\end{equation}
By Tate Duality (\cite[Corr. 2.4]{Milne}), we have an isomorphism
\begin{equation}
\mbox{H}^{1}(k,X(Z))\cong\mbox{Hom}(\mbox{H}^{1}(k,Z),\mathbb{C}^{\times}).\label{eq:5}
\end{equation}
 Using the isomorphism (\ref{eq:5}) in (\ref{eq:4-1}), we get a
map 
\begin{equation}
\hat{\gamma}_{\varphi}:\mbox{H}^{1}(k,Z)\rightarrow\widehat{R_{\varphi}},\label{eq:The map}
\end{equation}
 where $\widehat{R_{\varphi}}$ is the set of irreducible representations
of $R_{\varphi}$. Since $\Hh^{1}(k,X(Z))$ is abelian, the image
of $\hat{\gamma}_{\varphi}$ lies in the group of one dimensional
representations of $R_{\varphi}$.

\subsection{Statement of a conjecture}

Let $U$ be the unipotent radical of $B$ and let $p:G\rightarrow G_{\ad}:=G/Z$
be the adjoint morphism. We denote by the same symbol, the induced
map $p:T\rightarrow T_{\ad}:=T/Z$. 
\begin{defn}
\label{def:gen}A character $\psi:U(k)\rightarrow\mathbb{C}^{\times}$
is \textit{generic}\textbf{ }if its stabilizer in $T_{\ad}(k)$ is
trivial. 

The group $T_{\ad}(k)$ acts simply transitively on the set of generic
characters of $U(k)$. Hence the finite abelian group $T_{\ad}(k)/p(T(k))$
acts simply transitively on the set of $T(k)$-orbits of generic characters. 
\end{defn}

\begin{defn}
The \textit{pure inner forms} of $G$ are the groups $G^{\prime}$
over $k$ which are obtained by inner twisting by elements in the
pointed set $\Hh^{1}(k,G)$. 

All pure inner forms have the same center $Z$ over $k$. Let $G^{\prime}$
be a pure inner form of $G$. Denote the maximal torus of $G^{\prime}$
(resp. $G_{\ad}^{\prime}$) corresponding to $T$ (resp. $T_{\ad}$)
by $T^{\prime}$ (resp. $T_{\ad}^{\prime}$). We will denote the adjoint
morphism for all inner forms by the same symbol $p$. 
\end{defn}
We have a canonical inclusion $T_{\ad}^{\prime}(k)/p(T^{\prime}(k))\hookrightarrow\Hh^{1}(k,Z)$
and a canonical isomorphism $T_{\ad}^{\prime}(k)/p(T(k))\cong G_{\ad}^{\prime}(k)/p(G^{\prime}(k))$
(Lemma 5.1 \cite{DeRe10}). Equation (\ref{eq:The map}) thus induces
\[
\zeta_{\varphi}^{\prime}:G_{\ad}^{\prime}(k)/p(G^{\prime}(k))\rightarrow\widehat{R_{\varphi}.}
\]

Let $\tilde{\Pi}_{\varphi}$ denote the \textit{Vogan $L$-packet}
associate to $\varphi$. It is the union of the standard $L$-packets
associated to $\varphi$ of $G$ and all its pure inner forms. By
standard, we mean $L$-packets as defined in \cite{Borel}. Let $\rho\in\widehat{R_{\varphi}}\mapsto\pi_{\rho}\in\tilde{\Pi}_{\varphi}$
be the parametrization defined after the choice of a Whittaker datum.
Assume that this parametrization is compatible with Deligne's normalization
of the local Artin map (see \cite[Sec. 3]{ggp}). Let $\Pi_{\varphi}^{\prime}$
be the standard $L$-packet of $G^{\prime}$ contained in $\tilde{\Pi}_{\varphi}$.
The following is a conjecture in \cite[Sec. 9 (3)]{ggp}.
\begin{conjecture}
\label{conjec*}For $g\in G_{\ad}^{\prime}(k)$, $\pi_{\rho}\circ\Ad(g)=\pi_{g\cdot\rho}$,
where $g\cdot\rho=\rho\otimes\zeta_{\varphi}^{\prime}(g)$ and $\pi_{\rho}\in\Pi_{\varphi}^{\prime}$.
Thus $\pi_{\rho}$ is $\psi$-generic iff $\pi_{g\cdot\rho}$ is $g\cdot\psi$
generic. 
\end{conjecture}
We have a natural inclusion $\pi_{0}(\hat{Z}^{\Gamma_{k}})\subset\widehat{R_{\varphi}}$.
Let $\tau\in\widehat{R_{\varphi}}$. In \cite[Sec. 9(4)]{ggp}, it
is explained that the pure inner form of $G$ which acts on the representation
corresponding to the parameter $(\varphi,\tau)$ is determined by
the character $\tau|\pi_{0}(\hat{Z}^{\Gamma_{k}})$. Thus the standard
$L$-packet $\Pi_{\varphi}\subset\tilde{\Pi}_{\varphi}$ of $G$ is
parametrized by $\tau\in\widehat{R_{\varphi}}$ whose restriction
to $\pi_{0}(\hat{Z}^{\Gamma_{k}})$ is trivial. In other words, the
standard $L$-packet is parametrized by the irreducible representations
$\widehat{\mathcal{S}_{\varphi}}\hookrightarrow\widehat{R_{\varphi}}$
of the group $\mathcal{S}_{\varphi}:=\pi_{0}(Z_{\hat{G}}(\mbox{Im}(\varphi))/\hat{Z}^{\Gamma_{k}})$.
The map $\zeta_{\varphi}:G_{\ad}(k)/p(G(k))\rightarrow\widehat{R_{\varphi}}$
thus must factor through $\widehat{\mathcal{S}_{\varphi}}$. Conjecture
\ref{conjec*} for standard generic $L$-packets can be stated as:

\begin{conjbis}

\label{conjec} $\pi_{\rho}\in\Pi_{\varphi}$ is $\psi$-generic iff
$\pi_{g\cdot\rho}$ is $g\cdot\psi$ generic, where $\rho\in\widehat{R_{\varphi}}$
$g\in G_{\ad}(k)$, and where $g\cdot\rho=\rho\otimes\zeta_{\varphi}(g)$. 

\end{conjbis}
\begin{rem}
In \cite[Sec. 3]{kaletha}, Kaletha constructs a map $\zeta_{\varphi}:G_{\ad}(k)/p(G(k))\rightarrow\widehat{\mathcal{S}_{\varphi}}$.
In \cite[Sec. 1, eq. (1.1)]{kaletha}, he states the above conjecture
in a more precise manner by comparing the parametrization of a tempered
$L$-packet for different choices of Whittaker data. He also points
out that the action of $g\in G_{\ad}(k)$, should send $\rho\in\widehat{\mathcal{S}_{\varphi}}$
to $\rho\otimes\zeta_{\varphi}(g)$ or $\rho\otimes\zeta_{\varphi}^{-1}(g)$
depending on which of the two possible normalizations of the local
Artin map one chooses. The normalization in Conjecture \ref{conjec*}
uses Deligne's normalization \cite[Sec. 3]{ggp}. 
\end{rem}

\section{\label{sec:unrami_para}Description of $R$-group}

Let the notations be as in Section \ref{sec:Construction}. Assume
that $G$ is unramified, i.e., it is quasi-split and split over an
unramified extension of $k$. We also assume $k$ to be non-archimedean.
Let $I$ be the inertia subgroup of $W_{k}$ and let $\sigma$ be
the Frobenious element in $W_{k}/I$. Throughout this section, we
will abbreviate $\Hh^{1}(W_{k}/I,-)$ by the notation $\Hh^{1}(\sigma,-)$.

\subsection{Case of an unramified parameter}

Let $\bar{s}\in\hat{T}$ and let $\varphi$ be the Langlands parameter
determined by the map $\sigma\mapsto\bar{s}$. Let $s$ be a lift
of $\bar{s}$ in $\hat{T}_{\scon}\times\hat{\mathfrak{z}}$. 

Let $\mbox{H}^{1}(\sigma,\tilde{G})_{\sems}\subset\mbox{H}^{1}(\sigma,\tilde{G})$
denote the $\sigma$-conjugacy classes of the semisimple elements
of $\tilde{G}$, where $\tilde{G}=\hat{G}_{\scon}\times\hat{\mathfrak{z}}$
as in Section \ref{sec:Construction}. Denote by $[t]$, the class
of $t\in\tilde{G}_{\sems}$ in $\mbox{H}^{1}(\sigma,\tilde{G})_{\sems}$.
Let $A:=\pi_{1}(\hat{G})$. Let $\underline{A}$ denote $A_{\sigma}$,
the co-invariant of $A$ with respect to $\sigma$. We have $\mbox{H}^{1}(\sigma,A)\cong\underline{A}$.
Let $\underline{x}$ denote the image of $x\in A$ in $\underline{A}$.
Then there is an action of $\mbox{H}^{1}(\sigma,A)$ on $\mbox{H}^{1}(\sigma,\tilde{G})_{\sems}$
given by 
\[
\underline{x}\cdot[t]:=[xt]\qquad\mbox{ for }x\in A,t\in\tilde{G}_{\sems}.
\]
 Denote by $\underline{A}{}_{\varphi}$ the stabilizer of $[s]$ in
$\underline{A}$. 
\begin{lem}
\label{lem:R=00003DA}The map $\gamma_{\varphi}$ in equation \textnormal{(\ref{eq:4-1})}
induces an isomorphism $ $$R_{\varphi}\cong\underline{A}{}_{\varphi}$. \end{lem}
\begin{proof}
We have 
\begin{eqnarray*}
R_{\varphi}\cong\mbox{ker}(H^{1}(\sigma,A) & \rightarrow & \mbox{H}^{1}(\sigma,\tilde{G}))\\
 & = & \{\underline{x}\in\underline{A}|g^{-1}xs(^{\sigma}g)s^{-1}=1\qquad\mbox{ for some }g\in\tilde{G}\}\\
 & = & \{\underline{x}\in\underline{A}|\underline{x}\cdot[s]=[s]\}\\
 & = & \underline{A}{}_{\varphi}.
\end{eqnarray*}
\end{proof}
\begin{rem*}
The above Lemma is also proved in \cite{Yu}.
\end{rem*}

\subsection{\label{sub:Temp_para}Action of $\Omega$}

Let $\Psi=(X,R,\Delta,\check{X},\check{R},\check{\Delta})$ be the
based root datum of $(G,B,T)$. So $X$ (resp. $\check{X}$) is the
group of characters (resp. co-characters) of $T$, $R$ (resp. $\check{R}$)
is the set of roots (resp. co-roots) of $T$ in the Lie algebra of
$G$ and $\Delta$ (resp. $\check{\Delta}$) is a basis in $R$ (resp.
$\check{R}$) determined by $B$. Let $\underline{\Psi}=(\underline{X},\underline{R},\underline{\Delta},\check{\underline{X}},\check{\underline{R}},\check{\underline{\Delta}})$
be the based root datum obtained from $\Psi=(X,R,\Delta,\check{X},\check{R},\check{\Delta})$
by the construction given in \ref{subsub:Basedroot}. Let $\underline{Q}$
be the lattice generated by $\underline{R}$. Let $\underline{C}$
be the alcove in $\underline{V}:=\underline{X}\otimes\mathbb{R}$
determined by $\underline{\Delta}$. Let $W=W(\underline{\Psi})$
be the Weyl group of of the based root datum $\underline{\Psi}$.
By Theorem \ref{thm:yu}, it is the relative Weyl group of $G$. Let
$\underline{\Omega}\cong\underline{X}/\underline{Q}$ be the stabilizer
in $W\ltimes\underline{X}$ of $\underline{C}$ (see Section \ref{subsub:Omeg}).

By \cite[Lemma 6.5]{Borel} (or more directly by \cite[Prop. 11]{Mis2}),
we have 

\begin{equation}
\hat{T}_{\sigma}/W\cong(\hat{G}\rtimes\sigma)_{\sems}/\mbox{\textnormal{Int}}(\hat{G}),\label{mis2thm}
\end{equation}

where $(\hat{G}\rtimes\sigma)_{\sems}$ is the set of semisimple elements
in $\hat{G}\rtimes\sigma$ and $\mbox{\textnormal{Int}}(\hat{G})$
denotes the group of inner automorphisms of $\hat{G}$. 

Let $\hat{T}^{\cpt}$ be the maximal compact subtorus in $\hat{T}$.
Write $\hat{T}=X\otimes\mathbb{C}^{\times}$. Under this identification,
$\hat{T}^{\cpt}=X\otimes(\mathbb{R}/\mathbb{Z})\cong X\otimes\mathbb{R}/X$.
Let $\hat{G}^{\cpt}$ be the set of those semi-simple elements of
$\hat{G}$ which lie in some maximal compact subtorus of $\hat{G}$.
The isomorphism in (\ref{mis2thm}) induces an isomorphism $ $ 
\begin{eqnarray}
\hat{G}^{\cpt}\rtimes\sigma/\mbox{Int}(\hat{G}) & \cong & (\hat{T}^{\cpt})_{\sigma}/W\nonumber \\
 & \cong & \underline{X}\otimes\mathbb{\mathbb{R}}/W\ltimes\underline{X}\nonumber \\
 & = & \underline{X}\otimes\mathbb{\mathbb{R}}/((W\ltimes\underline{Q})\rtimes\underline{\Omega})\\
 & \longleftrightarrow & \bar{\underline{C}}/\underline{\Omega},\label{eq:omega1}
\end{eqnarray}

where $\bar{\underline{C}}$ is the closure of the alcove $\underline{C}$
determined by $\underline{\Delta}$.

Let $\hat{\mathfrak{z}}^{\cpt}:=X/Q\otimes\mathbb{R}$. It is the
Lie algebra of the maximal compact subtorus of $\hat{Z}$. Let $\tilde{G}^{\cpt}=\hat{G}_{\scon}^{\cpt}\times\hat{\mathfrak{z}}^{\cpt}$.
Then
\begin{eqnarray}
\tilde{G}^{\cpt}\rtimes\sigma/\mbox{Int}(\tilde{G}) & \cong & \tilde{T}_{\sigma}^{\cpt}/W\nonumber \\
 & \cong & ((\underline{X_{\scon}}\otimes(\mathbb{R}/\mathbb{Z}))\times(\underline{X}/\underline{Q}\otimes\mathbb{R}))/W\\
 & \cong & \underline{X}\otimes R/(\underline{Q}\rtimes W)\qquad\mbox{ since }\underline{X_{\scon}}=\underline{Q}\nonumber \\
 & \longleftrightarrow & \bar{\underline{C}}.\label{eq:omega2}
\end{eqnarray}
We have $\underline{A}\cong(X/Q)_{\sigma}\twoheadrightarrow\underline{X}/\underline{Q}\cong\underline{\Omega}$.
In Lemma \ref{lem:compatibility} below, we will show that the action
of $\underline{A}$ on $\tilde{G}^{\cpt}\rtimes\sigma/\mbox{Inn}(\tilde{G})\subset(\tilde{G}\rtimes\sigma)_{\sems}/\mbox{Inn}(\tilde{G})$
is compatible with the action of $\underline{\Omega}$ on $\bar{\underbar{C}}$.
Now $G$ is isogenous to $Z^{\circ}\times(G_{\scon})_{\der}$, where
$(G_{\scon})_{\der}$ is the simply connected cover of the derived
group of $G$ and $Z^{\circ}$ is the identity component of the center
of $G$. Since any simply connected semisimple group is the direct
product of almost simple groups, it suffices to prove the compatibility
in the case when $G$ is almost simple. 

Let $\underline{a}\in\underline{A}$ and let $a$ be a lift of $\underline{a}$
in $A$. Let $\underline{c_{0}}$ be the weighted barycenter of $\underline{C}$
and let $\underline{a}\mapsto\tilde{\omega}_{a}$ under the surjection
$\underline{A}\twoheadrightarrow\underline{\Omega}$, where $\omega_{a}\in W$
and $\tilde{\omega}_{a}$ is the affine transformation $x\in\underline{X}\otimes\mathbb{R}\mapsto\omega_{a}(x-\underline{c_{0}})+\underline{c_{0}}$
(see Section \ref{subsub:Omeg}). 

Let $[s]\mapsto x_{[s]}$ under the bijection $\tilde{G}^{\cpt}\rtimes\sigma/\mbox{Int}(\tilde{G})\leftrightarrow\bar{\underline{C}}$,
where $[s]$ denotes the class of $s\in\tilde{G}^{\cpt}$. Without
loss of generality, we can assume that $s\in\tilde{T}^{\cpt}$. 
\begin{lem}
\label{lem:compatibility}We have $\tilde{\omega}_{a}\cdot x_{[s]}=x_{[as]}$. \end{lem}
\begin{proof}
Let $\tilde{W}^{\circ}=W\ltimes\underline{Q}$. We have
\begin{eqnarray*}
\tilde{\omega}_{a}\cdot x_{[s]} & = & \omega_{a}(x_{[s]}-\underline{c_{0}})+\underline{c_{0}}\\
 & = & \omega_{a}\cdot x_{[s]}+(1-\omega_{a})\underline{c_{0}}\\
 & = & \omega_{a}(x_{[s]}+(\omega_{a}^{-1}-1)\underline{c_{0}}).
\end{eqnarray*}
 By Lemma \ref{lem:yu2}, $\tilde{\omega}_{a}\mapsto(\omega_{a}^{-1}-1)\underline{c_{0}}+\underline{Q}$
under the isomorphism $\underline{\Omega}\cong\underline{X}/\underline{Q}$.
Using this we get that $x_{[a]}\equiv(\omega_{a}^{-1}-1)\underline{c_{0}}$
mod $\tilde{W}^{\circ}$. Thus 
\begin{eqnarray*}
\tilde{\omega}_{a}\cdot x_{[s]} & \equiv & x_{[s]}+x_{[a]}\qquad\mbox{ mod }\tilde{W}^{\circ}\\
 & \equiv & x_{[as]}\qquad\mbox{ mod }\tilde{W}^{\circ}.
\end{eqnarray*}
Since $\tilde{\omega}_{a}\cdot x_{[s]}\in\underline{\bar{C}}$ and
$x_{[as]}\in\underline{\bar{C}}$, we conclude that 
\[
\tilde{\omega}_{a}\cdot x_{[s]}=x_{[as]}.
\]

\end{proof}

\subsection{Tempered parameter}

Let $\lambda$ be a unitary unramified character of $T(k)$. Let $\lambda\mapsto[\bar{s}]$
under the bijection 
\[
\mbox{Hom}(T(k),\mathbb{S}^{1})/W\cong\hat{G}^{\cpt}\rtimes\sigma/\mbox{Int}(\hat{G}),
\]
 where $\bar{s}$ can be chosen to be in $\hat{T}$ . Here $\mathbb{S}^{1}$
denotes the unit circle in $\mathbb{C}.$ Let $\varphi$ be the Langlands
parameter determined by the map $\sigma\mapsto\bar{s}$. Let $s$
be a lift of $\bar{s}$ in $\hat{T}_{\scon}\times\hat{\mathfrak{z}}$.

Let $\underline{\Omega}_{\varphi}$ be the stabilizer of $x_{[s]}\in\bar{\underline{C}}$
in $\underline{\Omega}$. We have
\begin{prop}
\label{thm:R-gp}$\mathcal{S}_{\varphi}\cong\underline{\Omega}{}_{\varphi}$.\end{prop}
\begin{proof}
By \cite[Lem. 2.5(iii)]{keys87}, $\pi_{0}(\hat{Z}^{\sigma})=\pi_{0}(\hat{T}^{\sigma})$.
But $\pi_{0}(\hat{T}^{\sigma})\cong(X_{\sigma})^{\tor}$. Since $\mathcal{S}_{\varphi}\cong R_{\varphi}/\pi_{0}(\hat{Z}^{\sigma})$,
by Lemma \ref{lem:R=00003DA} we get that $\mathcal{S}_{\varphi}\cong\underline{A}{}_{\varphi}/(X_{\sigma})^{\tor}\cong\underline{\Omega}{}_{\varphi}$.
Lemma \ref{lem:compatibility} shows that $\mathcal{S}_{\varphi}$
and $\underline{\Omega}_{\varphi}$ have compatible actions on $\tilde{G}^{\cpt}\rtimes\sigma/\mbox{Int}(\tilde{G})$
and $\bar{\underline{C}}$ respectively. 
\end{proof}

When $G$ is almost simple and simply connected, the non-trivial $\underline{\Omega}$
are given by the table below (see \cite[Sec. 9-4]{Kane} and \cite[Table-1]{reeder2008torsion}). 

\label{table}
\begin{table}[H]
\begin{tabular}{|c|c|}
\hline 
 & $\underline{\Omega}$\tabularnewline
\hline 
\hline 
$A_{n}$ & $\mathbb{Z}/(n+1)\mathbb{Z}$\tabularnewline
\hline 
$B_{n}$ & $\mathbb{Z}/2\mathbb{Z}$\tabularnewline
\hline 
$C_{n}$ & $\mathbb{Z}/2\mathbb{Z}$\tabularnewline
\hline 
\multirow{1}{*}{$D_{n}$ $(n\textnormal{ even})$ } & $\mathbb{Z}/2\mathbb{Z}\times\mathbb{Z}/2\mathbb{Z}$\tabularnewline
\hline 
$D_{n}$ $(n\textnormal{ odd})$  & $\mathbb{Z}/2\mathbb{Z}$\tabularnewline
\hline 
$E_{6}$ & $\mathbb{Z}/3\mathbb{Z}$\tabularnewline
\hline 
$E_{7}$ & $\mathbb{Z}/2\mathbb{Z}$\tabularnewline
\hline 
$\tensor*[^{2}]{A}{_{2n-1}}$ $(n\geq3)$ & $\mathbb{Z}/2\mathbb{Z}$\tabularnewline
\hline 
$\tensor*[^{2}]{D}{_{n+1}}$ $(n\geq2$) & $\mathbb{Z}/2\mathbb{Z}$\tabularnewline
\hline 
\end{tabular}

\caption{}
\end{table}

Let $R_{\lambda}$ be the Knapp-Stein $R$-group associated to $\lambda$
(see \cite[\S 2]{keys87} for definition). By \cite[Prop. 2.6]{keys87},
$R_{\lambda}\cong\mathcal{S}_{\varphi}$. Using Proposition \ref{thm:R-gp},
we obtain $\underline{\Omega}_{\varphi}\cong R_{\lambda}$. In fact,
the isomorphism is given by the restriction of the natural projection
$W\ltimes\underline{X}\rightarrow W$ to $\underline{\Omega}_{\varphi}$.
We get
\begin{thm}
\label{thm:mykeys}Let $G$ be an almost simple, simply connected,
unramified group defined over a non-archimedean local field $k$.
The non-trivial $R_{\lambda}$ that can appear are precisely the subgroups
of $\underline{\Omega}$ in table $1$.
\end{thm}
(see also \cite{KP13}).

This gives the classification obtained by Keys in \cite[\S 3]{keys}
in the case of unramified groups.

\section{\label{sec:packet_struct}Unramified $L$-packet }

Let the notations be as in Section \ref{sec:Construction}. Assume
further that $G$ is unramified and that $k$ is non-archimedean.
As in Section \ref{sec:unrami_para}, let $I$ be the inertia subgroup
of $W_{k}$ and let $\sigma$ be the Frobenious element in $W_{k}/I$.

An \textit{unramified $L$-packet} consists of those irreducible subquotients
of an unramified principal series representation of $G(k)$ which
have a non-zero vector fixed by some hyperspecial subgroup of $G(k)$.
Unramified $L$-packets are in bijective correspondence with $(\hat{G}\rtimes\sigma)_{\sems}/\mbox{\textnormal{Int}}(\hat{G})$.
Let $\varphi$ be a Langlands parameter determined by the $\sigma$-conjugacy
class of a semi-simple element and let $\Pi_{\varphi}$ be the associated
unramified $L$-packet. The $L$-packet $\Pi_{\varphi}$ is parametrized
by $\widehat{\mathcal{S}_{\varphi}}$, where $\mathcal{S}_{\varphi}:=\pi_{0}(Z_{\hat{G}}(\mbox{Im}(\varphi))/\hat{Z}^{\Gamma_{k}})$
as in Section \ref{sec:Construction}, after making the choice of
a hyperspecial point. We denote the parametrization by $\rho\in\widehat{\mathcal{S}_{\varphi}}\mapsto\pi_{\rho}\in\Pi_{\varphi}$.

Let $K$ be a compact subgroup of $G(k)$. Denote by $[K]$, the $G(k)$-conjugacy
class of $K$. If $\pi$ is a representation of $G(k)$, we denote
by $\pi^{K}$ the $K$-fixed points of the space realizing $\pi$.
By the notation $\pi^{[K]}\neq0$, we mean that $\pi$ has a non-zero
vector fixed by some (therefore any) conjugate of $K$. 

The conjugacy classes of hyperspecial subgroups of $G(k)$ form a
single orbit under $T_{\ad}(k)$. The author, in his Ph.D. thesis
\cite[Theorem 2.2.1]{Mis13} (also \cite[Thm. 1]{Mis1}) constructs
a map $T_{\ad}(k)/p(T(k))\twoheadrightarrow\widehat{\mathcal{S}_{\varphi}}$.
For the action of $T_{\ad}(k)$ on $\widehat{\mathcal{S}_{\varphi}}$
given by this map, he shows that $\pi_{t\cdot\rho}^{t\cdot[K]}\neq0\iff\pi_{\rho}^{[K]}\neq0$
for all $t\in T_{\ad}(t),$ $\rho\in\widehat{\mathcal{S}_{\varphi}}$,
where $K$ is a hyperspecial subgroup of $G(k).$ Using this result,
we have
\begin{thm}
\label{thm:Un_gen_struct}Let $\Pi_{\varphi}$ be an unramified L-packet
associated to a Langlands parameter $\varphi$. Then $\pi_{\rho}\in\Pi_{\varphi}$
is $\psi$-generic iff $\pi_{t\cdot\rho}$ is $t\cdot\psi$ generic
for all $t\in T_{\ad}(k).$ \end{thm}
\begin{proof}
Given $\pi_{\rho}\in\Pi_{\varphi}$, let $K$ be a hyperspecial subgroup
such that $\pi_{\rho}^{[K]}\neq0$. We can write $K$ as the stabilizer
$G(k)_{x}$ of some hyperspecial point $x$ in the Bruhat-Tits building
of $G(k)$. Without loss of generality we can assume $x$ to lie in
the apartment associated to $T$. We have that $(\textnormal{Ind}_{U(k)}^{G(k)}\psi)^{G(k)_{x}}\neq0$
iff there exists $g\in G$ such that $\psi|_{g^{-1}G(k)_{x}g\cap U(k)}\equiv1$.
Without loss of generality, we can assume that $g=1$. Let $t\in T_{\ad}(k).$
\begin{eqnarray*}
\Hom_{G(k)}(\pi_{\rho},(\textnormal{Ind}_{U(k)}^{G(k)}\psi))\neq0 & \mbox{iff} & (\textnormal{Ind}_{U(k)}^{G(k)}\psi)^{G(k)_{x}}\neq0\\
 & \mbox{iff} & \psi|_{G(k)_{x}\cap U(k)}\equiv1\\
 & \mbox{iff} & t\cdot\psi|_{G(k)_{t\cdot x}\cap U(k)}\equiv1\\
 & \mbox{iff} & (\mbox{Ind}_{U(k)}^{G(k)}t\cdot\psi)^{t\cdot[G(k)_{x}]}\neq0\\
 & \mbox{iff} & \Hom_{G(k)}(\pi_{t\cdot\rho},(\textnormal{Ind}_{U(k)}^{G(k)}t\cdot\psi))\neq0
\end{eqnarray*}
\end{proof}
\begin{rem}
Note that we do not assume $\varphi$ to be tempered. However, if
the associated $L$-packet is not generic, then the above statement
could be vacuous. 
\end{rem}

\begin{rem}
Theorem \ref{thm:Un_gen_struct} is a very special case of Conjecture
\ref{conjec}. In \cite[Thm. 3.3]{kaletha}, Kaletha proves Conjecture
\ref{conjec} for tempered representations in the case when $G$ is
a quasi-split real $K$-group or a quasi-split $p$-adic classical
group (in the sense of Arthur).
\end{rem}

\section*{Acknowledgment}

I am very thankful to Yakov Varshavsky for many helpful discussions.
I am also very thankful to Anne-Marie Aubert for her careful proof
reading and for pointing out the connection of this work with \cite{KP13}.
I am also grateful to David Kazhdan for hosting me at the Hebrew University
of Jerusalem, where this work was written. 

\bibliographystyle{alpha}
\bibliography{whitt}

\begin{thebibliography}{{Kar}11}

\bibitem[Art06]{arthur06}
J.~Arthur.
\newblock A note on {$L$}-packets.
\newblock {\em Pure Appl. Math. Q.}, 2(1, Special Issue: In honor of John H.
  Coates. Part 1):199--217, 2006.

\bibitem[AYY13]{AnYu}
J.~An, J.-K. Yu, and J.~Yu.
\newblock On the dimension datum of a subgroup and its application to
  isospectral manifolds.
\newblock {\em J. Differential Geom.}, 94(1):59--85, 2013.

\bibitem[BGA14]{Boro}
M.~Borovoi and C.~D. Gonz{\'a}lez-Avil{\'e}s.
\newblock The algebraic fundamental group of a reductive group scheme over an
  arbitrary base scheme.
\newblock {\em Cent. Eur. J. Math.}, 12(4):545--558, 2014.

\bibitem[Bor79a]{Borel}
A.~Borel.
\newblock Automorphic {$L$}-functions.
\newblock In {\em Automorphic forms, representations and {$L$}-functions
  ({P}roc. {S}ympos. {P}ure {M}ath., {O}regon {S}tate {U}niv., {C}orvallis,
  {O}re., 1977), {P}art 2}, Proc. Sympos. Pure Math., XXXIII, pages 27--61.
  Amer. Math. Soc., Providence, R.I., 1979.

\bibitem[Bor79b]{Borel1979}
A.~Borel.
\newblock Automorphic {$L$}-functions.
\newblock In {\em Automorphic forms, representations and {$L$}-functions
  ({P}roc. {S}ympos. {P}ure {M}ath., {O}regon {S}tate {U}niv., {C}orvallis,
  {O}re., 1977), {P}art 2}, Proc. Sympos. Pure Math., XXXIII, pages 27--61.
  Amer. Math. Soc., Providence, R.I., 1979.

\bibitem[DR10]{DeRe10}
S.~DeBacker and M.~Reeder.
\newblock On some generic very cuspidal representations.
\newblock {\em Compos. Math.}, 146(4):1029--1055, 2010.

\bibitem[GGP12]{ggp}
W.-T. Gan, B.~H. Gross, and D.~Prasad.
\newblock Symplectic root numbers, central critical values, and restriction
  problems in the representation theory of classical groups.
\newblock {\em Asterisque}, 2012.

\bibitem[Kal13]{kaletha}
T.~Kaletha.
\newblock Genericity and contragredience in the local {L}anglands
  correspondence.
\newblock {\em Algebra Number Theory}, 7(10):2447--2474, 2013.

\bibitem[Kan01]{Kane}
R.~Kane.
\newblock {\em Reflection groups and invariant theory}.
\newblock CMS Books in Mathematics/Ouvrages de Math\'ematiques de la SMC, 5.
  Springer-Verlag, New York, 2001.

\bibitem[{Kar}11]{Karpuk}
D.~A. {Karpuk}.
\newblock {Weil-\'etale Cohomology over $p$-adic Fields}.
\newblock {\em ArXiv e-prints}, November 2011.

\bibitem[Key82]{keys}
D.~Keys.
\newblock Reducibility of unramified unitary principal series representations
  of {$p$}-adic groups and class-{$1$} representations.
\newblock {\em Math. Ann.}, 260(4):397--402, 1982.

\bibitem[Key87]{keys87}
D.~Keys.
\newblock {$L$}-indistinguishability and {$R$}-groups for quasisplit groups:
  unitary groups in even dimension.
\newblock {\em Ann. Sci. \'Ecole Norm. Sup. (4)}, 20(1):31--64, 1987.

\bibitem[KP13]{KP13}
T.~Kamran and R.~Plymen.
\newblock {$K$}-theory and the connection index.
\newblock {\em Bull. Lond. Math. Soc.}, 45(1):111--119, 2013.

\bibitem[Kuo02]{Kuo02}
W.~Kuo.
\newblock Principal nilpotent orbits and reducible principal series.
\newblock {\em Represent. Theory}, 6:127--159 (electronic), 2002.

\bibitem[Kuo10]{Kuo10}
W.~Kuo.
\newblock The {L}anglands correspondence on the generic irreducible
  constituents of principal series.
\newblock {\em Canad. J. Math.}, 62(1):94--108, 2010.

\bibitem[Mil06]{Milne}
J.~S. Milne.
\newblock {\em Arithmetic duality theorems}.
\newblock BookSurge, LLC, Charleston, SC, second edition, 2006.

\bibitem[{Mis}12]{Mis1}
M.~{Mishra}.
\newblock {Structure of the Unramified L-packet}.
\newblock {\em ArXiv e-prints}, December 2012.

\bibitem[Mis13]{Mis13}
M.~Mishra.
\newblock {\em Structure of the unramified {L}-packet}.
\newblock ProQuest LLC, Ann Arbor, MI, 2013.
\newblock Thesis (Ph.D.)--Purdue University.

\bibitem[Mis15]{Mis2}
M.~Mishra.
\newblock Langlands parameters associated to special maximal parahoric
  spherical representations.
\newblock {\em Proc. Amer. Math. Soc.}, 143(5):1933--1941, December 2015.

\bibitem[Ree10]{reeder2008torsion}
M.~Reeder.
\newblock Torsion automorphisms of simple {L}ie algebras.
\newblock {\em Enseign. Math. (2)}, 56(1-2):3--47, 2010.

\bibitem[Ser97]{serre}
J.-P. Serre.
\newblock {\em Galois cohomology}.
\newblock Springer-Verlag, Berlin, 1997.
\newblock Translated from the French by Patrick Ion and revised by the author.

\bibitem[Sha90]{sha90}
F.~Shahidi.
\newblock A proof of {L}anglands' conjecture on {P}lancherel measures;
  complementary series for {$p$}-adic groups.
\newblock {\em Ann. of Math. (2)}, 132(2):273--330, 1990.

\bibitem[Spr79]{springer79}
T.~A. Springer.
\newblock Reductive groups.
\newblock In {\em Automorphic forms, representations and {$L$}-functions
  ({P}roc. {S}ympos. {P}ure {M}ath., {O}regon {S}tate {U}niv., {C}orvallis,
  {O}re., 1977), {P}art 1}, Proc. Sympos. Pure Math., XXXIII, pages 3--27.
  Amer. Math. Soc., Providence, R.I., 1979.

\bibitem[Yu]{Yu}
J.-K. Yu.
\newblock A note on the relative root datum of quasi-split groups (preprint).

\end{thebibliography}

\end{document}